\documentclass[a4paper,11pt]{amsart}

\newfont{\cyr}{wncyr10 scaled 1100}

\usepackage[usenames,dvipsnames]{color}

\usepackage[left=2.7cm,right=2.7cm,top=3.5cm,bottom=3cm]{geometry}
\usepackage{amsthm,amssymb,amsmath,amsfonts,mathrsfs,amscd,yfonts}
\usepackage{turnstile}
\usepackage[latin1]{inputenc}
\usepackage[all]{xy}
\usepackage{latexsym}
\usepackage{longtable}
\usepackage{graphicx}
\usepackage{hyperref}

\theoremstyle{plain}
\newtheorem{theorem}{Theorem}[section]

\newtheorem{lemma}[theorem]{Lemma}

\newtheorem{propo}[theorem]{Proposition}

\theoremstyle{definition}

\newtheorem{defi}[theorem]{Definition}

\newtheorem{examplewr}[theorem]{Example}
\newtheorem{ass}[theorem]{Assumption}

\theoremstyle{remark}
\newtheorem{obswr}[theorem]{Observation}
\newtheorem{remarkwr}[theorem]{Remark}

\newenvironment{remark}{\begin{remarkwr}\begin{upshape}}{\end{upshape}\end{remarkwr}}

\DeclareMathOperator{\dR}{\mathrm{dR}}

\DeclareMathOperator{\BK}{BK}

\DeclareMathOperator{\et}{et}

\DeclareMathOperator{\cyc}{cyc}
\DeclareMathOperator{\fin}{f}

\DeclareMathOperator{\loc}{loc}
\DeclareMathOperator{\Gr}{Gr}

\DeclareMathOperator{\quo}{quo}
\DeclareMathOperator{\sub}{sub}
\DeclareMathOperator{\TSym}{TSym}
\DeclareMathOperator{\fami}{fam}

\newcommand{\cW}{\mathcal W}

\newcommand{\Q}{\mathbb{Q}}
\newcommand{\Z}{\mathbb{Z}}

\newcommand{\C}{\mathbb{C}}

\newcommand{\Gal}{\mathrm{Gal\,}}

\newcommand{\Fil}{\mathrm{Fil}}

\newcommand{\Frob}{\mathrm{Fr}}

\newcommand{\Fr}{\mathrm{Fr}}

\newfont{\gotip}{eufb10 at 12pt}

\newcommand{\cO}{{\mathcal O}}

\newcommand{\ra}{\rightarrow}
\newcommand{\lra}{\longrightarrow}

\newcommand{\res}{\mathrm{res}}

\newcommand{\fp}{{\mathfrak p}}

\include{thebibliography}

\begin{document}

\title[Congruences among Euler systems]{Eisenstein congruences among Euler systems}

\author{\'Oscar Rivero and Victor Rotger}

\begin{abstract}
We investigate Eisenstein congruences between the so-called Euler systems of Garrett--Rankin--Selberg type. This includes the cohomology classes of Beilinson--Kato, Beilinson--Flach and diagonal cycles. The proofs crucially rely on different known versions of the Bloch--Kato conjecture, and are based on the study of the Perrin-Riou formalism and the comparison between the different $p$-adic $L$-functions.
\end{abstract}





\address{O. R.: Simons Laufer Mathematical Sciences Institute, 17 Gauss Way, Berkeley, CA 94720, United States of America}
\email{riverosalgado@gmail.com}

\address{V. R.: IMTech, UPC and Centre de Recerca Matem\`{a}tiques, C. Jordi Girona 1-3, 08034 Barcelona, Spain }
\email{victor.rotger@upc.edu}


\maketitle

\tableofcontents

\section{Introduction}

The theory of Euler systems is a powerful machinery that provides deep insight on the arithmetic of Galois representations, Selmer groups and Iwasawa main conjectures. There are relatively few instances of such Euler systems in the literature, and their construction is far from being a systematic issue. This work focuses on the interaction between different kinds of Euler systems, and can be seen as a natural continuation of our previous article \cite{RiRo3}, where we had studied congruence relations between the system of circular units and Beilinson--Kato elements.

This note is concerned with the theory of Eisenstein congruences and the development of a {\it $p$-adic Artin formalism} at the level of Euler systems. The easiest instance of an Euler system is that of circular units, obtained as a weighted combination of cyclotomic units after applying the Kummer map. Its image under the Perrin-Riou regulator (a.\,k.\,a.\,Coleman map) allows us to recover the Kubota--Leopoldt $p$-adic $L$-function.

The three Euler systems we want to discuss here are slightly harder to describe. They are sometimes referred as {\it Euler systems of Garrett--Rankin--Selberg type}, and the most well-known example comes from Kato's work \cite{Kato}, which allowed the proof of the Bloch--Kato conjecture for elliptic curves defined over $\mathbb Q$ in rank 0 and of one of the divisibilities of the Iwawawa main conjecture. The other two instances come from generalizations of that previous construction, when the two modular units involved in Kato's construction are replaced by cuspidal forms. The case of Beilinson--Flach classes has been explored in different works of Bertolini, Darmon and Rotger \cite{BDR2}, and Kings, Lei, Loeffler and Zerbes \cite{LLZ}, \cite{KLZ2}, \cite{KLZ}. The case of diagonal cycles, where a proper Euler system is not available (but a family of cohomology classes varying over Hida or Coleman families) has been worked out by Darmon and Rotger \cite{DR2}, \cite{DR3} and by Bertolini, Seveso and Venerucci \cite{BSV-munster}, \cite{BSV}.

\subsection{The set-up}

Let $N, k > 1$ be positive integers, $\chi_f: (\Z/N\Z)^\times \ra \bar\Q^\times$ a Dirichlet character and $f \in S_k(N,\chi_f)$ a normalized cuspidal eigenform of level $N$, weight $k$ and nebentype $\chi_f$. Fix a prime $p\nmid 6N \varphi(N)$. Similarly, we fix eigenforms $g \in S_{\ell}(N, \chi_g)$ and $h \in S_m(N, \chi_h)$, with $\ell,m \geq 2$ and $\chi_g,\chi_h$ Dirichlet characters. Let $F$ be the finite extension of $\Q$ generated by the field of coefficients of $f$, $g$, $h$ and the values of all Dirichlet characters of conductor $N$; let $\mathcal O$ be its ring of integers and $\mathfrak p \subset \cO$ a prime ideal above $p$. Before continuing, we need to introduce the following objects:
\begin{itemize}
\item[-] $\mathscr H_{\mathbb Z_p}$ is the \'etale $\mathbb Z_p$-sheaf on the open modular curve $Y_1(N)$ given by the Tate module of the universal elliptic curve $\mathcal E/Y_1(N)$, as introduced in \cite[Def. 3.1.1]{KLZ} or \cite[\S2.3]{KLZ2}.
\item[-] $\TSym^k \mathscr H_{\mathbb Z_p}$ is the sheaf of symmetric tensors of degree $k-2$ over $\mathscr H_{\mathbb Z_p}$ defined in \cite[Def. 3.1.2]{KLZ} or \cite[\S2.2]{KLZ2}; after inverting $(k-2)!$ it is isomorphic to the $(k-2)$-th symmetric power.
\item[-] $T_{f,Y}$ is the integral $p$-adic Galois representations given as the $f$-isotypical quotient of  $H^1_{\et}(\bar{Y},\TSym^{k-2}(\mathscr H_{\Z_p})(1)) \otimes_{\mathbb Z_p} \mathcal O_{\mathfrak p}$ of the closed modular curve of level $\Gamma_1(N)$, as defined in \cite[eq. (11)]{RiRo3} in the case of trivial coefficients; see also \eqref{defTf}.
\item[-] $T_{f,X}$ is defined in the same way as $T_{f,Y}$ (see \cite[eq. (11)]{RiRo3}), but using compactly supported cohomology instead. We set, as usual, $V_{f,Y}=T_{f,Y}\otimes \Q$ and $V_{f,X}=T_{f,Y}\otimes \Q$, and note that it holds $V_{f,Y}=V_{f,X}$.
\end{itemize}

We also attach analogous lattices to $g$ and $h$. All along the article, we impose that $k+\ell+m$ is even and that $\chi_f \chi_g \chi_h = 1$. We assume that the triple $(k,\ell,m)$ is balanced, that is, $\ell+m \geq k$, $m+k \geq \ell$ and $k+\ell \geq m$. This means that we are in the so-called {\it geometric} region, where a construction of cohomology classes in an appropriate $H^1$ is meaningful and the corresponding $L$-value vanishes. For simplicity, we also assume $\chi_h \not\equiv 1 \pmod{\mathfrak p}$.

Along this short note we deal with three different objects, whose characterizations and main properties are recalled in Section \S\ref{sec:back}. According to the discussions of \cite[Rk. 3.3]{BSV}, the global classes we introduce may have some bounded denominators, but along the introduction we assume that they are integral at $p$, which is always the case provided that $p$ is large enough (in particular, we need that $p>\max\{k,\ell,m\}$ to invert the factorials appearing in the formulas of loc.\,cit.); we come back to this issue later on in the text. 

Set $c = (k+\ell+m-2)/2$ and consider:
\begin{enumerate}
\item[(a)] The Beilinson--Kato class $\kappa_f \in H^1(\Q, T_{f,Y}(1)(\ell+m-c-1)$, whose characterization is recalled in Section \ref{sec:kato} following \cite{Kato} and \cite{BD}.
\item[(b)] The Beilinson--Flach class $\kappa_{f,g} \in H^1(\Q, T_{f,Y} \otimes T_{g,Y}(m-c))$; we recall its main properties in Section \ref{sec:bf} following \cite{BDR1}, \cite{BDR2} and \cite{KLZ}.
\item[(c)] The diagonal cycle class $\kappa_{f,g,h} \in H^1(\Q, T_{f,Y} \otimes T_{g,Y} \otimes T_{h,Y}(1-c))$, introduced in Section \ref{sec:cycles} following \cite{DR1}, \cite{DR2} and \cite{BSV-munster}.
\end{enumerate}

Moreover, and as a piece of notation, we denote $\bar \kappa_f$, $\bar \kappa_{f,g}$, $\bar \kappa_{f,g,h}$ the reduction of the above integral classes modulo a suitable power of the prime ideal $\mathfrak p$ (usually denoted by $\mathfrak p^t$).

\subsection{Main results}

We use the notations about Eisenstein series introduced in \cite{RiRo3}; in particular, they are indexed by the weight and by a pair of Dirichlet characters whose products of nebentypes is $N$. When $g$ is congruent modulo $\mathfrak p^t$ with the Eisenstein series $E_{\ell}(\chi_g,1)$, as we discussed in loc.\,cit., there is a projection map \[ T_{g,Y} \otimes \mathcal O/\mathfrak p^t \longrightarrow \mathcal O/\mathfrak p^t(\chi_g). \] Hence, the class $\bar \kappa_{f,g}$ projects to an element \[ \bar \kappa_{f,g,1} \in H^1(\Q, T_{f,Y} \otimes \mathcal O/\mathfrak p^t(\chi_g)(m-c)). \] When this projection vanishes modulo $\mathfrak p^t$, and proceeding as in \cite[\S4]{RiRo3}, we may lift $\kappa_{f,g}$ to a class in $H^1(\mathbb Q, T_{f,Y} \otimes T_{g,X})$ and consider instead the projection \[ T_{g,X} \otimes \mathcal O/\mathfrak p^t \longrightarrow \mathcal O/\mathfrak p^t(\ell-1). \] Thus, we obtain a class \[ \bar \kappa_{f,g,2} \in H^1(\Q, T_{f,Y}(\ell+m-c-1) \otimes \mathcal O/\mathfrak p^t). \]

To state our first theorem, let $L_p(f,\chi_g,s)$ stand for the Mazur--Swinnerton-Dyer $p$-adic $L$-function attached to the pair $(f,\chi_g)$, with the conventions of \cite[\S3]{RiRo3}. The result depends on a {\it weak Gorenstein assumption} (properly introduced as Assumption \ref{ass:goren}) and on the conditions introduced as Assumption \ref{ass:bf0} and Assumption \ref{ass:bf1}. The latter makes reference to the conditions of big image and non-exceptional zeros which give a Bloch--Kato type result, which allows us to conclude that the corresponding spaces on which the classes live are one-dimensional. As discussed in the corresponding section, these results are known to hold in many different instances, so the assumptions are quite mild in nature.

\begin{theorem}
Suppose that Assumption \ref{ass:goren} holds for $g$, that the Galois representation $T_{f,Y}$ is absolutely irreducible, and that Assumptions \ref{ass:bf0} and \ref{ass:bf1} also hold.

If $L_p(f,\chi_g,m-c+1) \neq 0 \pmod{\mathfrak p^t}$, then $\bar \kappa_{f,g,1} = 0$ and the following congruence holds in $H^1(\mathbb Q, T_{f,Y}(\ell+m-c-1) \otimes \mathcal O/\mathfrak p^t)$: \[ \bar \kappa_{f,g,2} = \bar \kappa_f. \]
\end{theorem}

The proof is inspired by the ideas introduced in our previous work, using the Perrin-Riou formalism to perform a comparison at the level of $p$-adic $L$-functions. However, here, we have to proceed in a slightly different way: the first natural projection is 0 because of the local condition satisfied by the Beilinson--Flach class, and it is at second order where we can do a proper comparison.

We now mimic that approach in the framework of diagonal cycles. When $h$ is congruent modulo $\mathfrak p^t$ with $E_m(\chi_h,1)$ we may define a class \[ \bar \kappa_{f,g,h,1} \in H^1(\Q, T_{f,Y} \otimes T_{g,Y} \otimes \mathcal O/\mathfrak p^t(\chi_h)(1-c)). \] When it vanishes, and proceeding as in the previous case, we have the refinement \[ \bar \kappa_{f,g,h,2} \in H^1(\Q, T_{f,Y} \otimes T_{g,Y}(m-c) \otimes \mathcal O/\mathfrak p^t). \]

Let $L_p(f,g,\chi_h,c)$ (resp. $L_p(g,f,\chi_h,c)$) stand for the value at $s=c$ of the Rankin--Selberg $p$-adic $L$-function attached to $(f,g,\chi_h)$ with $f$ dominant (resp. $g$ dominant), as recalled in Section \ref{sec:bf}.

\begin{theorem}
Suppose that Assumption \ref{ass:goren} holds for $h$, that the Galois representations $T_{f,Y}$ and $T_{g,Y}$ are absolutely irreducible, and that Assumptions \ref{ass:cycles0} and \ref{ass:cycles1} hold.

If $L_p(f,g,\chi_h,c) \neq 0 \pmod{\mathfrak p^t}$ or $L_p(g,f,\chi_h,c) \neq 0 \pmod{\mathfrak p^t}$, then $\bar \kappa_{f,g,h,1} = 0$ and the following congruence holds in $H^1(\mathbb Q, T_{f,Y} \otimes T_{g,Y}(m-c) \otimes \mathcal O/\mathfrak p^t)$: \[ \bar \kappa_{f,g,h,2} = \bar{\kappa}_{f,g}. \]
\end{theorem}

We believe that most of the results we discuss are amenable to generalizations to Hida or Coleman families, but in that case one must be more circumspect about integrality issues. Deformations to weight one of the above congruence formulae may be particularly interesting, given their potential connection to the equivariant Birch and Swinnerton-Dyer conjecture. Furthermore, it would be interesting to explore whether our results could have applications to non-vanishing results, as in \cite[\S3]{Va}, or to the study of classical problems like Goldfeld's conjecture, as in \cite{KrLi}, where they establish congruences between elliptic units and Heegner points using a comparison between the $p$-adic $L$-functions of Katz and Bertolini--Darmon--Prasanna.

We finally point out that some of the results we state contain several assumptions which are likely to be relaxed, but we have not pursued this in an attempt to keep this work short and to avoid technical complications in that regard.

\subsection{Relation to other works}

This note grew up as an attempt to better understand the interactions among Euler systems coming from Eisenstein congruences with those discussed by Loeffler and the first author \cite{LR}, where Coleman families passing through the critical $p$-stabilization of an Eisenstein series are studied. In that setting, we obtain analogous relations (in characteristic zero) between Beilinson--Flach and Beilinson--Kato classes, and between Beilinson--Kato and Beilinson--Flach. Compare \cite[Thms. C1.13, C2.11]{LR} in loc.\,cit. with the main result of this paper. The reason behind this parallelism is that in both scenarios we have two natural filtrations (one corresponding to the usual \'etale cohomology of the modular curve and the other to \'etale cohomology with compact support) whose interplay is used in an analogous way. To be more precise, in this article we consider the lattices attached to the open and closed modular curve, that we denote by $T_{f,Y}$ and $T_{f,X}$, respectively. In this case, there is an inclusion $T_{f,Y} \supset T_{f,X} \supset I \cdot T_{f,Y} \supset I \cdot T_{f,X} \supset \ldots$, where $I$ is the Eisenstein ideal. This is a parallel situation to what we had described in the discussion after \cite[Cor. 6.1.3]{LR}, where the maximal ideal $\mathfrak m$ considered in loc.\,cit. plays the role of the Eisenstein ideal. We hope that the interactions between these two settings can be pushed forward in subsequent works.

Another recent work dealing with the connections between different kinds of Euler systems is the article of Bertolini--Darmon--Venerucci \cite{BDV}, where they prove a conjecture of Perrin-Riou relating Heegner points and Beilinson--Kato elements via the use of Beilinson--Flach elements. We believe that this kind of techniques can have multiple applications to arithmetic problems and to the construction of new Euler systems, and we expect to come back to these issues in forthcoming work.

\subsection{Acknowledgments and funding}

We thank Kazim B\"uy\"ukboduk and David Loeffler for stimulating comments and discussions around the topics studied in this note.

This project has received funding from the European Research Council (ERC) under the European Union's Horizon 2020 research and innovation programme (grant agreement No 682152). O.R. was further supported by the Royal Society Newton International Fellowship NIF\textbackslash R1\textbackslash 202208. V. R. is supported by Icrea through an Icrea Academia Grant. This material is based upon work supported by the National Science Foundation under Grant No. DMS-1928930 while the first author was in residence at the Mathematical Sciences Research Institute in Berkeley, California, during the Spring 2023 semester.

\section{Background: Euler systems and reciprociy laws}\label{sec:back}

\subsection{Modular curves and lattices}\label{sec:lattice}

We begin by recalling some notations about modular curves and lattices. For that purpose, fix algebraic closures $\bar\Q$, $\bar\Q_p$ of $\Q$ and $\Q_p$ respectively, and  embeddings of $\bar\Q$ into $\bar\Q_p$ and $\C$. The former singles out a prime ideal $\mathfrak p$ of $\cO$ lying above $p$ and we let $\mathcal O_{\mathfrak p}$ denote the completion of $\mathcal O$ at $\mathfrak p$. We also fix throughout an uniformizer $\varpi$ of $\mathcal O_{\mathfrak p}$ and an isomorphism $\mathbb C_p \simeq \mathbb C$.

Given a variety $Y/\Q$ and a field extension $F/\Q$, let $Y_F = Y \times F$ denote the base change of $Y$ to $F$ and set $\bar Y = Y_{\bar\Q}$. Fix an integer $N\geq 3$ and let $Y_1(N) \subset X_1(N)$ denote the canonical models over $\mathbb Q$ of the (affine and projective, respectively) modular curves classifying pairs $(A,i)$ where $A$ is a (generalized) elliptic curve and $i: \mu_N \rightarrow A$ is an embedding of group schemes.

Let $k \geq 2$ be an integer and $\theta$ a Dirichlet character. For an ordinary newform $f\in S_k(N,\theta)$ satisfying the congruence $f \equiv E_k(\theta,1) \,\, \mathrm{mod} \, \mathfrak p^t$, let $f(q)= \sum a_n(f) q^n$ denote its $q$-expansion at the cusp $\infty$. Enlarge $F$ so that it also contains the eigenvalues $\{ a_n(f)\}_{n\geq 1}$, and still write $\cO$ for its ring of integers and $\cO_{\fp}$ for its completion at the prime ideal $\mathfrak{p}$ over $p$. For our further use, we denote by $\alpha_f\in \cO_{\mathfrak p}^\times$ (resp. $\beta_f \in \cO_{\mathfrak p}$) the unit root (resp. non-unit root) of the $p$-th Hecke polynomial of $f$. Let $\psi_f: G_{\Q_p} \lra \cO_{\mathfrak p}^\times$ denote the unramified character characterized by $\psi_f(\Frob_p)=\alpha_f$.

Let $I_f^* = (T^*_\ell-a_\ell(f)) \subset \mathbb{T}^*$ denote the ideal associated to the system of eigenvalues of $f$ with respect to the dual Hecke operators. As in the introduction, consider the $\cO_{\fp}$-module
\begin{equation}\label{defTf}
T_{f,Y} = H^1_{\et}(\bar{Y},\TSym^{k-2}(\mathscr H_{\Z_p})(1)) \otimes_{\mathbb Z_p} \mathcal O_{\mathfrak p}/ I^*_f,
\end{equation}
and define similarly $T_{f,X}$, using compact support cohomology instead.

Note that the two lattices $T_{f,X}$ and  $T_{f,Y}$ may give rise to different $\cO_{\fp}[G_{\Q}]$-modules in spite of the fact that the associated rational Galois representations \[ V_f := T_{f,X} \otimes F_{\fp} \simeq T_{f,Y} \otimes F_{\fp} \] are isomorphic.

Finally, let $\Sigma_X$ and $\Sigma_Y$ denote the torsion submodules of $T_{f,X}$ and $T_{f,Y}$ respectively, and let
\begin{equation}\label{Tf-free}
T_{f,Y,\circ}^{\quo} := T_{f,Y}^{\quo}/\Sigma_Y \simeq \cO_{\mathfrak p}(\psi_f)
\end{equation}
denote the free quotient.

Since $f \equiv E_k(\theta,1) \,\mbox{mod}\,\mathfrak p^t$, we have
\begin{equation}\label{Vfcong}
T_{f,Y,\circ}^{\quo} \otimes \mathcal O/\mathfrak p^t \simeq \mathcal O/\mathfrak p^t(\theta),
\end{equation}
which amounts to the congruence $\psi_f \equiv \theta \, (\mbox{mod}\,\mathfrak p^t)$ as unramified characters of $G_{\Q_p}$.

In the cases where we deal with several modular forms of common level $N$, we slightly abuse notations and still write $F$ for the finite extension of $\mathbb Q$ generated by the field of coefficients of all of the them and the values of all Dirichlet characters of conductor $N$. Similarly, $\mathcal O$ stands for the ring of integers and $\mathfrak p \subset \mathcal O$ for a prime ideal above $p$.

As in our previous work, we need the following assumption to establish our main results.
\begin{ass}\label{ass:goren}
The $G_{\mathbb Q_p}$-module $\mathcal O/\mathfrak p^t(\theta)$ does not show up as a quotient of $\Sigma_Y/\Sigma_X$.
\end{ass}

In this case, we will say that $T_{f,Y}$ (or simply $f$, for short) satisfies the {\it weak-Gorenstein condition.} The reason is that this assumption automatically follows if the localization of the Hecke algebra acting on $M_k(\Gamma_1(N))$ at the Eisenstein ideal is Gorenstein.

\vskip 12pt

Several of our results are suitably formulated in the language of families. Let $\mathcal W$ denote the weight space defined as the formal spectrum of the Iwasawa algebra $\Lambda=\mathcal O_{\mathfrak p}[[\mathbb Z_p^{\times}]]$. The set of classical points in $\cW$ is given by characters $\nu_{s,\xi}$ of the form $\xi \varepsilon_{\cyc}^{s-1}$ where $\xi$ is a Dirichlet character of $p$-power conductor, $\varepsilon_{\cyc}$ is the cyclotomic character and $s$ is an integer; this forms a dense subset in $\cW$ for the Zariski topology. Let $\mathcal W^{\circ}$ further denote the set of those points with $\xi=1$; we shall often write $s$ in place of $\nu_s = \nu_{s,1}$. Let $\mathcal W^{\pm} \subset \mathcal W$ denote the topological closure of the set of points $\xi \varepsilon_{\cyc}^{s-1}$ with $(-1)^{s-1} \xi(-1) = \pm 1$. We have $\cW= \cW^+ \sqcup \cW^-$ and we write $\Lambda = \Lambda^+ \oplus \Lambda^-$ for the corresponding decomposition of the Iwasawa algebra.

\subsection{Congruences between canonical periods}\label{sec:periods}


Recall that the de Rham Dieudonn\'e module $D_{\dR}(V_f)$ associated to the eigenform $f$ is a $F_{\mathfrak p}$-filtered vector space of dimension $2$. Poincar\'e duality yields a perfect pairing \[ \langle \, , \, \rangle: D_{\dR}(V_f(-1)) \times D_{\dR}(V_{f^*}) \rightarrow F_{\mathfrak p} \] and there is an exact sequence of Dieudonn\'e modules
\begin{equation}\label{Ds}
0 \rightarrow D_{\dR}(V_f^{\sub}) \rightarrow D_{\dR}(V_f) \rightarrow D_{\dR}(V_f^{\quo}) \rightarrow 0,
\end{equation}
where $D_{\dR}(V_f^{\sub})$ and $D_{\dR}(V_f^{\quo})$ have both dimension $1$.

Falting's theorem associates to $f$ a regular differential form $\omega_f \in \Fil(D_{\dR}(V_f))$, which gives rise to an element in $D_{\dR}(V_f^{\quo})$ via the right-most map in \eqref{Ds} and in turn induces a linear form
\begin{equation}\label{omegaf}
\omega_f \, : \, D_{\dR}(V_{f^*}^{\sub}(-1)) \rightarrow F_{\mathfrak p}, \quad \eta \mapsto \langle \omega_f, \eta \rangle
\end{equation}
that we continue to denote with same symbol by a slight abuse of notation.

There is also a differential $\eta_f$, which is characterized by the property that it spans the line $D_{\dR}(V_f^{\sub}(-1))$ and is normalized so that
\begin{equation}\label{def-etaf}
\langle \eta_f, \omega_{f^*} \rangle = 1.
\end{equation}
Again, it induces a linear form
\begin{equation}\label{etaf0}
\eta_f \, : \, D_{\dR}(V_{f^*}^{\quo}) \rightarrow F_{\mathfrak p}, \quad \omega \mapsto \langle \eta_f, \omega \rangle
\end{equation}

Now, let $T$ be an unramified $\cO_{\mathfrak p}[\Gal(\bar\Q_p/\Q_p)]$-module and set $V=T\otimes F_{\mathfrak p}$. Let $\hat\Z_p^\mathrm{ur}$ denote the completion of the ring of integers of the maximal unramified extension of $\Q_p$, and define the integral Dieudonn\'e module
$$
D(T) := (T \,  \hat\otimes_{\Z_p}  \, \hat\Z_p^\mathrm{ur})^{\Fr_p=1}.
$$
As shown in loc.\,cit.\,we have $D_{\dR}(V) = D(T)\otimes F_{\mathfrak p}$.

As explained e.g.\,in \cite[Prop.\,1.7.6]{FK}, there is a functorial isomorphism of $\cO_{\mathfrak p}$-modules (forgetting the Galois structure) given by
\begin{equation}\label{D(T)=T}
T  \, \stackrel{\sim}{\lra } D(T).
\end{equation}
This map is not canonical as it depends on a choice of root of unity; for $T=\cO_{\mathfrak p}(\chi)$ we take it to be given by the rule $1 \mapsto \, \mathfrak g(\chi)$.

Recall the free $\cO_{\mathfrak{p}}$-quotient
$$
T_{f,Y,\circ}^{\quo}  \simeq \cO_{\mathfrak p}(\psi_f);
$$
note in particular that $T_{f,Y,\circ}^{\quo}$ is unramified.

Recall the uniformizer $\varpi$ fixed at the outset in \S \ref{sec:lattice}. Let $C_{f^*} =(\varpi^r)\subseteq O_{\mathfrak p}$ denote the congruence ideal attached to $f^*$ as defined e.g.\,in \cite{Oh3}. In Hida's terminology, $\varpi^r$, is sometimes called a congruence divisor.
According to \cite[\S 10.1.2]{KLZ2}, the image of $\eta_{f^*|D(T_{f,Y}^{\quo})}$ is precisely the inverse $C_{f^*}^{-1}$ of the congruence ideal, and hence there is an isomorphism
\[ \eta_{f^*} \, : \, D(T_{f,Y,\circ}^{\quo}) \longrightarrow C_{f^*}^{-1}, \quad \omega \mapsto \langle \eta_{f^*},\omega \rangle. \] Setting $\tilde \eta_{f^*} := \varpi^r \cdot \eta_{f^*}$, the above map gives rise to an isomorphism of $O_{\mathfrak p}$-modules
\begin{equation}\label{etaf}
\tilde \eta_{f^*} \, : \, D(T_{f,Y,\circ}^{\quo}) \longrightarrow \mathcal O_{\mathfrak p}, \quad \omega \mapsto \langle \tilde \eta_{f^*},\omega \rangle.
\end{equation}

Choose an isomorphism of $G_{\mathbb Q_p}$-modules.
\begin{equation}\label{bar-iota}
\bar\iota_1: \bar{T}_{f,Y,\circ}^{\quo} \stackrel{\sim}{\lra} \cO/\mathfrak{p}^t(\theta).
\end{equation}
Fixing such a map amounts to choosing the class $(\mathrm{mod}\,{\mathfrak p^t})$ of an isomorphism of local modules $\iota_1: T_{f,Y,\circ}^{\quo}  \simeq \mathcal O_{\mathfrak p}(\psi_f)$.
In light of the functoriality provided by \eqref{D(T)=T} this determines and is determined by the class $(\mathrm{mod}\,{\mathfrak p^t})$ of an isomorphism $D(\iota_1): D(T_{f,Y,\circ}^{\quo}) \simeq D(\cO_{\mathfrak p}(\psi_f))$.

We choose $\iota_1$ as the single isomorphism  making the following diagram commutative:
\begin{equation}\label{rig-iso}
\xymatrix
{ D(T_{f,Y,\circ}^{\quo})  \quad  \ar[r]^{\quad \langle \, , \tilde \eta_{f^*}\rangle}\ar[d]^{D(\iota_1)} & \quad  \mathcal O_{\mathfrak p}  \\
 D(\cO_{\mathfrak p}(\psi_f))  \ar[ur]_{\cdot 1/ \mathfrak g(\theta)}&
 }
\end{equation}
Indeed, since both $ \langle \, , \tilde \eta_{f^*}\rangle$ and $\cdot 1/ \mathfrak g(\theta)$ are isomorphisms, it follows that such a map $D(\iota_1)$ exists and is unique, and this in turn pins down $\iota_1$ and $\bar\iota_1$ in light of \eqref{D(T)=T}.

There are similar results for the differential $\omega_f$. Using again \cite[\S10.1.2]{KLZ2}, there is an isomorphism
\[ \omega_{f^*} \, : \, D(T_{f,Y}^{\sub}) \longrightarrow \mathcal O_{\mathfrak p}, \quad \eta \mapsto \langle \omega_{f^*}, \eta \rangle. \]

Hence, we may choose an isomorphism of $G_{\mathbb Q_p}$-modules
\begin{equation}\label{bar-iota2}
\bar\iota_2: \bar{T}_{f,Y}^{\sub} \stackrel{\sim}{\lra} \cO/\mathfrak{p}^t(1).
\end{equation}
Fixing such a map amounts to choosing the class $(\mathrm{mod}\,{\mathfrak p^t})$ of an isomorphism of local modules $\iota_2: T_{f,Y}^{\sub}  \simeq \mathcal O_{\mathfrak p}(k-1)$. This determines and is determined by the class $(\mathrm{mod}\,{\mathfrak p^t})$ of an isomorphism $D(\iota_2): D(T_{f,Y}^{\sub}) \simeq D(\cO_{\mathfrak p})$.

We choose $\iota_2$ as the single isomorphism  making the following diagram commutative:
\begin{equation}\label{rig-iso2}
\xymatrix
{ D(T_{f,Y}^{\sub})  \quad  \ar[r]^{\quad \langle \, , \omega_{f^*}\rangle}\ar[d]^{D(\iota_2)} & \quad  \mathcal O_{\mathfrak p}  \\
 D(\cO_{\mathfrak p})  \ar[ur]_{\cdot 1}&
 }
\end{equation}
Indeed, since both $ \langle \, , \omega_{f^*}\rangle$ and $\cdot 1$ are isomorphisms, it follows that such a map $D(\iota_2)$ exists and is unique, and this in turn pins down $\iota_2$ and $\bar\iota_2$.

\subsection{The Beilinson--Kato Euler system}\label{sec:kato}

In this section we recall the reciprocity law for Beilinson--Kato elements of arbitrary weight. For details on the constructions, see e.g. \cite{Kato} or \cite{BD}. As in the introduction, let $f \in S_k(N,\chi_f)$, and consider two auxiliary positive integers $(\ell,m)$, with $\ell,m \geq 2$ and the triple $(k,\ell,m)$ balanced. The Beilinson--Kato class is an element
\begin{equation}\label{class:bk}
\kappa_f \in H^1(\mathbb Q, T_{f,Y}(\psi)(1))
\end{equation}
obtained as the cup product of two Eisenstein classes $\delta_g$ and $\delta_h$ of weights $\ell-2$ and $m-2$, respectively. In particular, \[ \delta_g \in H_{\et}^1(\bar Y, \TSym^{\ell-2}(\mathscr H_{\mathbb Z_p})(1)) \quad \text{and} \quad \delta_h \in H_{\et}^1(\bar Y, \TSym^{m-2}(\mathscr H_{\mathbb Z_p})(1)). \] As a piece of notation, we write $c = (k+\ell+m-2)/2$ and $q = (\ell+m-k)/2 > 0$. Note that the classes may also depend on an auxiliary Dirichlet character $\chi$. However, when the choices of $\ell$, $m$ and $\chi$ are clear from the setting, we drop them from the notation as in \eqref{class:bf}. Note that the previous class is integral at $p$ provided that $p \nmid N$, as it follows from the discussions of e.g. \cite[\S5.4]{Kings} and \cite[\S4,7]{KLZ}.

It may be instructive to keep in mind the parallelism with the weight two scenario considered in \cite[\S2]{RiRo3}. In that case, where $k=\ell=m=2$, the Beilinson--Kato class is obtained as the cup product of the modular units \[ u_{1,\chi} \cup u_{1,\xi} \in H_{\et}^2(\bar Y,\mathbb Z_p(2)), \] where $u_{1,\chi} \in H^1(\mathbb Q, \mathbb Z_p(1))$, $u_{1,\xi} \in H^1(\mathbb Q, \mathbb Z_p(1))$. Applying now the Hochschild--Serre map and projecting to the $f$-isotypical component, we get the desired class.

As a piece of notation, let $$\underline{\varepsilon}_{\cyc}: G_{\mathbb Q} \rightarrow \Lambda^{\times}$$ denote the $\Lambda$-adic cyclotomic character which sends a Galois element $\sigma$ to the group-like element $[\varepsilon_{\cyc}(\sigma)]$. It interpolates the powers of the $\mathbb Z_p$-cyclotomic character, in the sense that for any arithmetic point of the form $\nu_{s,\xi}\in \cW$, \begin{equation}\label{lambda-eps-cyc}
\nu_{s,\xi} \circ \underline{\varepsilon}_{\cyc} = \xi \cdot \varepsilon_{\cyc}^{s-1}.
\end{equation}

\begin{propo}\label{perrin1}
Assume that $\alpha_f \psi(p) \not\equiv 1 \pmod{\mathfrak p}$. Then there exists a homomorphism of $\Lambda$-modules
\[
\mathcal L_{f}^-: H^1(\mathbb Q_p, T_{f,Y,\circ}^{\quo} \otimes \Lambda(\psi)(\varepsilon_{\cyc}^q \underline{\varepsilon}_{\cyc})) \longrightarrow \Lambda
\]
satisfying the following interpolation property: for $s \in \mathcal W^{\circ}$, the specialization of $\mathcal L_{f}^-$ at $s$ is the homomorphism
\[
\mathcal L_{f,s}^-: H^1(\mathbb Q_p, T_{f,Y,\circ}^{\quo}(\psi)(s+q-1)) \longrightarrow \mathcal O_{\mathfrak p}
\]
given by
\[
\mathcal L_{f,s}^- = \frac{1- \bar \psi(p) \alpha_{f}^{-1} p^{-s}}{1-\psi(p)\alpha_f p^{s-1}} \times \begin{cases} \frac{(-1)^s}{(s-1)!} \times \langle \log_{\BK}, t^s \tilde \eta_{f^*} \rangle & \text{ if } s \geq 2-q \\ (-s)! \times \langle \exp_{\BK}^*, t^s \tilde \eta_{f^*} \rangle & \text{ if } s < 2-q, \end{cases}
\]
where $\log_{\BK}$ is the Bloch-Kato logarithm and $\exp_{\BK}^*$, the dual exponential map.
\end{propo}

\begin{proof}
This follows from Coleman and Perrin-Riou's theory of $\Lambda$-adic logarithm maps as extended by Loeffler and Zerbes in \cite{LZ}. In particular, \cite[Theorem 8.2.3]{KLZ2} yields an injective map
$$
H^1(\mathbb Q_p, T_{f,Y,\circ}^{\quo} \otimes \Lambda(\psi)(\varepsilon_{\cyc}^q \underline{\varepsilon}_{\cyc})) \lra D(T_{f,Y,\circ}^{\quo}) \otimes \Lambda,
$$
since $H^0(\mathbb Q_p,T_{f,\circ}^{\quo}(\psi)) = 0$ because of the assumption that $\alpha_f \psi(p) \not\equiv 1$ modulo $\mathfrak p$. This map is characterized by the interpolation property formulated in \cite[Appendix B]{LZ}. Next we apply the pairing of \eqref{etaf} and the result follows.
\end{proof}

In particular, for $q=0$, we recover the homomorphism introduced in \cite[\S3.4]{RiRo3}. The assumption $\alpha_f \psi(p) \not \equiv 1$ modulo $\mathfrak p$ is necessary to have image in $\Lambda$, but it may be still possible that for a fixed value of $s$ the image is integral even though the condition is not satisfied (see \cite[\S8.2]{KLZ2} for a precise description of the kernel and cokernel of that map).

The following result, essentially due to Kato (and reformulated in our language by Bertolini and Darmon), shows that there exists a $\Lambda$-adic cohomology class whose bottom layer is precisely $\kappa_f$ and whose image under a suitable regulator map recovers the $p$-adic $L$-function attached to $f^*$. To fix notations, let $L_p(f^*,\psi,s)$ stand for the Mazur--Swinnerton-Dyer $p$-adic $L$-function attached to a modular form $f^*$ and a Dirichlet character $\psi$, with the conventions used in Section 3 of loc.\,cit.. Let $L_p(f^*,E_k(\chi_1,\chi_2),s)$ be the Hida--Rankin $p$-adic $L$-function attached to the two modular forms $f^*$ and $E_k(\chi_1,\chi_2)$.

\begin{propo}\label{kato-law}
Under the assumptions of Proposition \ref{perrin1}, there exists a $\Lambda$-adic cohomology class \[ \kappa_{f,\infty} \in H^1(\mathbb Q, T_{f,Y} \otimes \Lambda^-(\varepsilon_{\cyc}^q \underline{\varepsilon}_{\cyc})) \] such that:
\begin{enumerate}

\item[(a)] There is an explicit reciprocity law
\[ \mathcal L_{f}^-(\res_p(\kappa_{f,\infty})^-) = L_p(f^*, E_{\ell}(1,\chi),q+s), \] and there is a factorization \[ L_p(f^*, E_{\ell}(1,\chi),q+s) = \frac{-L_p(f^*,\chi,\frac{k+\ell-m}{2})}{\mathfrak g(\bar \chi)} \times L_p(f^*,q+s). \]
Here, $\res_p$ stands for the map corresponding to localization at $p$ and $\res_p(\kappa_{f,\infty})^-$ is the map induced in cohomology from the projection map $T_{f,Y} \rightarrow T_{f,Y,\circ}^{\quo}$.

\item[(b)]  The bottom layer $\kappa_{f}(q)$ lies in $H_{\fin}^1(\mathbb Q, T_{f}(q))$ and satisfies \[ \kappa_{f}(q) = \mathcal E_f \cdot \kappa_f, \] where $\mathcal E_f$ is the Euler factor \[ \mathcal E_f = (1- \bar \chi(p) \beta_{f}  p^{m-c-1})(1-\bar \chi_f \chi(p) \beta_{f} p^{\ell-c-1})(1- \alpha_{f}p^{-1})(1-\beta_{f}p^{-1}). \]
\end{enumerate}
\end{propo}

\begin{proof}
This is due to Kato \cite{Kato}; with the current formulation, see more specifically the work of Bertolini and Darmon \cite[Theorems 4.4 and 5.1]{BD}.
\end{proof}

\subsection{The Beilinson--Flach Euler system}\label{sec:bf}

We now introduce the Beilinson--Flach classes and a reciprocity law for them. See \cite{KLZ} and \cite{KLZ2} for details on the construction. Let $f \in S_f(N,\chi_f)$ and $g \in S_{\ell}(N,\chi_g)$ be two cuspidal eigenforms with $\chi_f \chi_g \not\equiv 1 \pmod{\mathfrak p}$, and let $s$ be an integer such that $1 \leq s < \min(k,\ell)$. The Beilinson--Flach element constructed in \cite{KLZ} is a class
\begin{equation}\label{class:bf}
\kappa_{f,g} \in H^1(\mathbb Q, T_{f,Y} \otimes T_{g,Y}(1-s)),
\end{equation}
obtained via the pushforward of a class \[ \delta_h \in H_{\et}^1(\bar Y, \TSym^{m-2}(\mathscr H_{\mathbb Z_p})(1)) \] along the diagonal $Y \hookrightarrow Y^2$. Here, we can see the cyclotomic twist $s$ as a shortcut for $c-m+1$. Observe once more that these classes have no denominators at $p$; see \cite[Rk. 7.1.2]{KLZ} and \cite[\S3.2]{KLZ2}.

In the case where $k=\ell=2$ and $s=1$, the construction can be summarized as follows: take the Eisenstein element $u_{1, \chi_f \chi_g} \in H_{\et}^1(Y, \mathbb Z_p(1))$ and consider the Gysin map (push-forward of appropriate sheaves) along the diagonal $Y \hookrightarrow Y \times Y$ to get a class in $H_{\et}^3(Y^2,\mathbb Z_p(2))$. Applying the Hochschild--Serre map and projecting to the $(f,g)$-istoypic component we get the desired class. See e.g. \cite[\S4]{KLZ}. (It must be pointed out that we are allowed to use the Hochschild--Serre identification since the $H^3$ of a surface over $\bar{\mathbb Q}$ vanishes.)


\begin{remark}
The classes we are working with are the so-called {\it Eisenstein classes} in the terminology of \cite[\S8]{KLZ2}. Observe that the Euler factors that appear at Proposition 8.1.3 of loc.\,cit. are precisely the same as those of the previous section (except for the $c$-factor that arises in the reciprocity law for Beilinson--Flach elements when we vary $f$ and $g$ in families and which plays no role when $\chi_f \chi_g \neq 1$).
\end{remark}

\begin{propo}\label{perrin2}
Assume that $\alpha_f \alpha_g^{-1} \chi_g(p) \not \equiv 1 \pmod{\mathfrak p}$. Then there exists a homomorphism of $\Lambda$-modules \[ \mathcal L_{f,g}^{-+}: H^1(\mathbb Q_p, T_{f,Y,\circ}^{\quo} \otimes T_{g,Y}^{\sub} \otimes \Lambda(\varepsilon_{\cyc}\underline{\varepsilon}_{\cyc})) \longrightarrow \Lambda \] satisfying the following interpolation property: for $r \in \mathcal W^{\circ}$, the specialization of $\mathcal L_{f,g}^-$ at $s$ is the homomorphism \[ \mathcal L_{f,g,s}^{-+}: H^1(\mathbb Q_p, T_{f,Y,\circ}^{\quo} \otimes T_{g,Y}^{\sub} (s)) \longrightarrow \mathcal O_{\mathfrak p} \] given by \[ \mathcal L_{f,g,s}^{-+} = \frac{1-\alpha_f^{-1} \beta_g^{-1} p^{-s}}{1-\alpha_f \beta_g p^{s-1}} \times \begin{cases} \frac{(-1)^{s+1}}{s!} \times \langle \log_{\BK}, t^{s+1} \tilde \eta_{f^*} \otimes \omega_{g^*} \rangle & \text{ if } s \geq 0 \\ (-s-1)! \times \langle \exp_{\BK}^*, t^{s+1} \tilde \eta_{f^*} \otimes \omega_{g^*} \rangle & \text{ if } s < 0. \end{cases} \]
\end{propo}

\begin{proof}
As in Proposition \ref{perrin1}, this follows from the general theory of Perrin-Riou maps (see \cite[Thm. 8.2.8]{KLZ2} for this formulation).
\end{proof}

For the following result, let $L_p(f^*, g^*,s)$ stand for the Hida--Rankin $p$-adic $L$-function attached to the convolution of two modular forms. Note that this function is not symmetric on the two modular forms. Let $\chi = \bar \chi_f \bar \chi_g$. The following result, established by Kings, Loeffler and Zerbes, shows that there exists a $\Lambda$-adic class whose bottom layer is $\kappa_{f,g}$ (up to a suitable Euler factor) and which recovers the Hida--Rankin $p$-adic $L$-functions.

\begin{propo}\label{bf-law}
Under the assumptions of Proposition \ref{perrin2}, there exists a $\Lambda$-adic cohomology class \[ \kappa_{f,g,\infty} \in H^1(\mathbb Q, T_{f,Y} \otimes T_{g,Y} \otimes \Lambda(\varepsilon_{\cyc} \underline{\varepsilon}_{\cyc})) \] such that:
\begin{enumerate}
\item[(a)] There is an explicit reciprocity law
\[ \mathcal L_{f}^-(\res_p(\kappa_{f,g,\infty})^{-+}) = L_p(f^*, g^*,1+s), \]
where $\res_p$ stands for the map corresponding to localization at $p$ and $\res_p(\kappa_{f,g,\infty})^{-+}$ is the map induced in cohomology from the projection map $T_{f,Y} \otimes T_{g,Y} \rightarrow T_{f,Y,\circ}^{\quo} \otimes T_{g,Y}^{\sub}$.

\item[(b)]  The bottom layer $\kappa_{f,g}(s) \in H^1(\mathbb Q, T_{f,Y} \otimes T_{g,Y}(s))$ satisfies \[ \kappa_{f,g}(s) = \mathcal E_{f,g} \cdot \kappa_{f,g}, \] where $\mathcal E_{f,g}$ is the Euler factor \[ \mathcal E_{f,g} = (1-\beta_f \alpha_g p^{-s}) (1-\beta_f \beta_g \chi(p) p^{s+1-k-\ell})(1-\beta_g \alpha_f p^{-s}) (1-\beta_g \beta_f p^{-s}). \]
\end{enumerate}
\end{propo}

\begin{proof}
This follows from considering the result of \cite[Thm. 10.2.2]{KLZ2} when the two modular forms are fixed.
\end{proof}

In particular, for our applications we do not need to consider the whole cyclotomic variation, and it is enough to consider our attention to {\it geometric} twists, that is, values of $s$ for which certain inequalities are satisfied (and in particular, the results of \cite{KLZ} suffice).

\subsection{Diagonal cycles}\label{sec:cycles}

We finally introduce the cohomology classes obtained via diagonal cycles and their reciprocity laws, following \cite{DR3}, \cite{BSV-munster} and \cite[\S3]{BSV}. We also assume that the classes we are going to deal with have no denominators at $p$, and following the discussion of \cite[Rk. 3.3]{BSV} this is always the case provided that $p$ is greater than the weights. Consider a triple of cuspidal eigenforms $f \in S_k(N,\chi_f)$, $g \in S_{\ell}(N,\chi_g)$ and $h \in S_m(N,\chi_h)$, with $k+\ell+m$ even and $\chi_f \chi_g \chi_h = 1$; we also assume that the Galois representation attached to $f$ is absolutely irreducible. Let $c=(k+\ell+m-2)/2$. The diagonal cycle class is an element
\begin{equation}\label{class:diag} \kappa_{f,g,h} \in H^1(\mathbb Q, T_{f,Y} \otimes T_{g,Y} \otimes T_{h,Y}(1-c)).
\end{equation}
When the weights are all two, the class is constructed by pushing forward the class $1 \in H_{\et}^0(Y, \mathbb Z_p)$ along the diagonal inclusion $Y \hookrightarrow Y^3$; this gives an element in $H^4(Y^3,\mathbb Z_p(2))$, which gives the desired class after applying the Hochschild--Serre map following by projection to the $(f,g,h)$-isotypic component. Following \cite{DR1}, one can also view the diagonal class directly over the closed modular curve, which in our case is a consequence of the congruence relations we will develop (at least modulo $p$).

Contrary to the previous settings, the absence of a cyclotomic variable makes the situation slightly different. Fortunately, we only need the reciprocity law pointwise, at a point of weight $(k,\ell,m)$. If a greater level of generality were needed, we are allowed to move all three modular forms $(f,g,h)$ along a Hida (or Coleman) family, but restricting to the central twist. Although the Perrin-Riou map makes sense for the whole family of twists, since the classes are only available for the central ones, we restrict to those ones from now on.

\begin{defi}\label{perrin3}
Assume that $\alpha_f \chi_f^{-1}(p) \alpha_g^{-1} \alpha_h^{-1} \not \equiv 1$ modulo $\mathfrak p$. Then we define $\mathcal L_{f,g,h}^{-++}$ as the morphism of $\mathbb Z_p$-modules \[ \mathcal L_{f,g,h}^{-++}: H^1(\mathbb Q_p, T_{f,Y,\circ}^{\quo} \otimes T_{g,Y}^{\sub} \otimes T_{h,Y}^{\sub}(1-c)) \longrightarrow \mathcal O_{\mathfrak p} \] given by \[ \mathcal L_{f,g,h}^{-++} = \frac{1-\alpha_f^{-1} \beta_g^{-1} \beta_h^{-1} p^{-1}}{1-\alpha_f \beta_g \beta_h} \times  \langle \log_{\BK}, \tilde \eta_{f^*} \otimes \omega_{g^*} \otimes \omega_{h^*} \rangle \]
\end{defi}

For the following result, let $L_p(f^*, g^*, h^*,c)$ stand for the special value at $s=c$ of the triple product $p$-adic $L$-function associated to the Hida families passing through $(f,g,h)$, as introduced in \cite[\S4]{DR1}. As recalled in loc.\,cit., we can also interpret these values in terms of a suitable pairing, from which it is clear the parallelism with the Hida--Rankin one used in previous sections.

When dealing with diagonal cycles, we implicitly fix a choice of test vectors, for instance as in \cite{Hs} (note that we can do it since the dominant $p$-adic $L$-function will never be the Eisenstein one). To keep the parallelism with the previous settings, let \[ \kappa_{f,g,h,\fami} = \mathcal E_{f,g,h} \cdot \kappa_{f,g,h}, \] where $\mathcal E_{f,g,h}$ is the Euler factor \[ \mathcal E_{f,g,h} = (1-\beta_f \alpha_g \beta_h p^{-c}) (1-\beta_f \beta_g \alpha_h p^{-c})(1-\alpha_f \beta_g \beta_h p^{-c}) (1-\beta_f \beta_g \beta_h p^{-c}). \] We can understand this class as the result of the specialization of a class varying in Hida families.

\begin{propo}\label{diag-law}
Under the conditions of Definition \ref{perrin3}, there is an equality
\[ \mathcal L_{f,g,h}^{-++}(\res_p(\kappa_{f,g,h,\fami})^{-++}) = L_p(f^*, g^*, h^*,c), \]
where $\res_p$ stands for the map corresponding to localization at $p$ and $\res_p(\kappa_{f,g,h,\fami})^{-++}$ is the map induced in cohomology from the projection map $T_{f,Y} \otimes T_{g,Y} \otimes T_{h,Y}(1-c) \rightarrow T_{f,Y,\circ}^{\quo} \otimes T_{g,Y}^{\sub} \otimes T_{h,Y}^{\sub}(1-c)$.
\end{propo}

\begin{proof}
This follows from the reciprocity law for diagonal cycles as in e.g. \cite[Thm. 1.3]{DR1}; see also \cite[\S3]{BSV}.
\end{proof}

\section{Congruences between Beilinson--Flach and Beilinson--Kato classes}

Let $f \in S_k(N,\chi_f)$ and $g \in S_{\ell}(N,\chi_g)$ be two cuspidal eigenforms with $\chi_f \chi_g \not\equiv 1 \pmod{\mathfrak p}$. We also fix an integer $s$ with $1 \leq s < \min\{k,\ell\}$, and introduce the auxiliary variable $m=k+\ell-2s$. Recall that $F$ is a finite extension of $\mathbb Q$ containing the field of coefficients of $f$ and $g$, and $\mathcal O$ is its ring of integers. Fix algebraic closures $\bar\Q$, $\bar\Q_p$ of $\Q$ and $\Q_p$ and an embedding $\bar\Q \hookrightarrow \bar\Q_p$. This singles out a prime ideal $\mathfrak p$ of $\cO$ lying above $p$ and we let $\mathcal O_{\mathfrak p} \subset \bar\Q_p$ denote the completion of $\mathcal O$ at $\mathfrak p$. The Beilinson--Flach class we consider here is the element \[ \kappa_{f,g} \in H^1(\mathbb Q, T_{f,Y} \otimes T_{g,Y}(1-s)) \] introduced in Section \ref{sec:bf}, where we omit the dependence on $s$.

In this section, we carry out a comparison with the Beilinson--Kato element \[ \kappa_f \in H^1(\mathbb Q, T_{f,Y} \otimes \mathcal O/\mathfrak p^t (\ell-s)) \] of Section \ref{sec:kato}.

Throughout all the section, we keep the following simplifying assumptions, which are implicitly used in the different results we develop. In case of needing some extra condition for a particular claim, we will state it explicitly.
\begin{itemize}
\item The {\it weak-Gorenstein condition} for $g$, as presented in Assumption \ref{ass:goren}.
\item The residual Galois representation attached to $f$ is absolutely irreducible.
\end{itemize}


Consider now the Eisenstein congruence setting, where \[ g \equiv E_{\ell}(\chi_g,1) \pmod{\mathfrak p^t}, \] or alternatively $g^* \equiv E_{\ell}(1,\bar \chi_g)$, which implies that $\mathfrak p^t$ divides the Bernoulli number $B_2(\bar \chi_g)$.

Since $g \equiv E_{\ell}(\chi_g,1)$, we may consider the projection $T_{g,Y} \otimes \mathcal O/\mathfrak p^t \longrightarrow \mathcal O/\mathfrak p^t(\chi_g)$, that gives rise to a class \[ \bar \kappa_{f,g,1} \in H^1(\mathbb Q, T_{f,Y} \otimes \mathcal O/\mathfrak p^t(\chi_g)(1-s)). \]

When $\bar \kappa_{f,g,1}$ vanishes, and following the arguments of \cite[Prop. 4.1]{RiRo3}, we may consider a putative refinement of it, namely \[ \bar \kappa_{f,g,2} \in H^1(\mathbb Q, T_{f,Y} \otimes \mathcal O/\mathfrak p^t(\ell-s)). \]

The objective of this section is studying these cohomology classes and relate them with the arithmetic of Beilinson--Kato classes. Our study rests on the Perrin-Riou theory, which require suitable bounds on the corresponding Selmer groups coming from Iwasawa theory.

We state the assumptions that are needed to establish our results. We write $H_{\Gr}^1$ for the $H^1$ with Greenberg local conditions, resulting from imposing the unramified condition at the primes $v \neq p$ and the Bloch--Kato condition at $p$ (see e.g. \cite[\S11]{KLZ}).

We introduce now two different sets of assumptions, that are needed for the two different results that will be established.

\begin{ass}\label{ass:bf0}
\begin{enumerate}
\item[(a)] The space $H_{\Gr}^1(\mathbb Q, T_{f,Y}(\chi_g)(1-s))$ is one-dimensional.
\item[(b)] It holds that \[ H^0(\mathbb Q, T_{f,Y}(\chi_g)(1-s) \otimes \mathcal O/\mathfrak p^t) = 0, \quad H^2(\mathbb Q, T_{f,Y}(\chi_g)(1-s)) = 0 \] and $\alpha_f \chi_g(p) \not \equiv 1 \pmod{\mathfrak p}$.
\end{enumerate}
\end{ass}

\begin{ass}\label{ass:bf1}
\begin{enumerate}
\item[(a)] The space $H_{\fin}^1(\mathbb Q, T_{f,Y}(\ell-s))$ is one-dimensional.
\item[(b)] It holds that $H^2(\mathbb Q, T_{f,Y}(\ell-s)) = 0$ and $\alpha_f \not \equiv 1 \pmod{\mathfrak p}$.
\end{enumerate}
\end{ass}

Kato \cite{Kato} proved that under big image assumptions for the representation $T_{f,Y}$, condition (a) is true, so this can be seen as a mild condition (but in particular we want to exclude the case where $f$ is also congruent to an Eisenstein series). Assumption (b) may be regarded as a condition regarding non-exceptional zeros which also assures that the denominator of the Perrin-Riou map does not vanish modulo $\mathfrak p$.

\begin{remark}
The triviality of both the local $H^0$ or the local $H^2$ may be formulated in terms on a condition over the Hecke eigenvalues. If $p>k+\ell-2$, these are the following ones:
\begin{itemize}
\item The group $H^0(\mathbb Q_p, T_{f,Y}(\chi_g)(1-s) \otimes \mathcal O/\mathfrak p^t) \neq 0$ if and only if $s=1$ and $\alpha_f \chi_g(p) \equiv 1 \pmod{\mathfrak p}$.
\item The group $H^2(\mathbb Q_p, T_{f,Y}(\chi_g)(1-s)) \neq 0$ if and only if $s=1$ and $\chi_g(p) \equiv 1 \pmod{\mathfrak p}$.
\item The group $H^2(\mathbb Q, T_{f,Y}(\ell-s)) \neq 0$ if and only if $s=\ell-1$ and $\alpha_f(p) \equiv 1 \pmod{\mathfrak p}$.
\end{itemize}
If $p>k+\ell-2$ and under Assumption \ref{ass:goren}, $H^0(\Q_p, T_{f,Y}(\ell-s) \otimes \mathcal O/\mathfrak p^t) = 0$. (Observe however that the global $H^0$ may be zero even though the corresponding local group is non-trivial.)

Finally, note that if $p \leq k + \ell - 2$ we must be more cautious: $H^0(\mathbb Q_p, \mathcal O/\mathfrak p(r))$ is non-zero when $r \equiv 0$ modulo $p-1$ (the cyclotomic character pulls back to the Teichm\"uller lift). See \cite[\S3.1]{Man} for a more elaborated discussion on that issue.
\end{remark}

\begin{propo}
Under the conditions of Assumptions \ref{ass:bf0} (resp. \ref{ass:bf1}), the Selmer group with Greenberg condition $H_{\Gr}^1(\mathbb Q, T_{f,Y}(\chi_g)(1-s) \otimes \mathcal O/\mathfrak p^t)$ (resp. $H_{\Gr}^1(\mathbb Q, T_{f,Y}(\ell-s) \otimes \mathcal O/\mathfrak p^t)$) is one-dimensional.
\end{propo}

\begin{proof}
We do the proof for the case of $T_{f,Y}(\ell-s)$, and the other case is completely analogous. Beginning with \[ 0 \longrightarrow T_{f,Y}(\ell-s) \xrightarrow{\times p} T_{f,Y}(\ell-s) \longrightarrow T_{f,Y}(\ell-1)/pT_{f,Y}(\ell-s) \longrightarrow 0, \] our assumptions guarantee that there is a short exact sequence \[ 0 \rightarrow H^1(\mathbb Q, T_{f,Y}(\ell-s)) \xrightarrow{\times p} H^1(\mathbb Q, T_{f,Y}(\ell-s)) \rightarrow H^1(\mathbb Q, T_{f,Y}(\ell-s) \otimes \mathcal O/\mathfrak p^t) \rightarrow 0, \] and therefore \[ \frac{H^1(\mathbb Q, T_{f,Y}(\ell-s))}{pH^1(\mathbb Q,T_{f,Y}(\ell-s))} \simeq H^1(\mathbb Q, T_{f,Y}(\ell-s) \otimes \mathcal O/\mathfrak p^t(\ell-s)). \] The conclusion now follows.
\end{proof}

\begin{lemma}\label{lemma:inj-kato}
The Bloch--Kato logarithm of Proposition \ref{perrin1} is injective.
\end{lemma}
\begin{proof}
Note that $H^0(\mathbb Q_p, T_{f,Y}(\ell-s))=0$ by an immediate inspection of the Hodge--Tate weights. Then the result follows from the construction of the Bloch--Kato logarithm discussed in \cite[Section 8.2]{KLZ}.
\end{proof}

The first result of this note establishes the vanishing of $\loc_p(\bar \kappa_{f,g,1})$ under the previous assumptions. Note however that there may be interesting cases where the hypotheses fail, leading to a tantalizing connection with the theory of circular units that we expect to explore in forthcoming work.

\begin{propo}
Suppose the conditions of Assumption \ref{ass:bf0} are true. If $L_p(f,\chi_g,s) \neq 0$ modulo $p$, then the class $\bar \kappa_{f,g,1}$ vanishes.
\end{propo}

\begin{proof}
Using the running assumptions, the space $H_{\Gr}^1(\mathbb Q, T_{f,Y}(\chi_g)(1-s) \otimes \mathcal O/\mathfrak p^t)$ is one-dimensional and contains a canonical element, the Beilinson--Kato class $\kappa_f$ resulting from the specialization of the element $\kappa_{f,\infty}$ in Proposition \ref{kato-law}.

The projection of this class to $H^1(\mathbb Q_p, T_{f,Y}^-(\chi_g)(1-s) \otimes \mathcal O/\mathfrak p^t)$ is non-zero, since its image under the Perrin-Riou regulator is precisely $L_p(f,\chi_g,s)$, which does not vanish due to the assumptions. However, the projection to that quotient of $\bar \kappa_{f,g,1}$ is zero because of the local properties satisfied by the Beilinson--Flach class (see \cite[Lemma 8.1.5, Prop. 8.1.7]{KLZ}).
\end{proof}

When $\bar \kappa_{f,g,1}$ vanishes modulo $\mathfrak p$, we may consider the class $\bar \kappa_{f,g,2}$, and compare it with the previous Kato class.

The comparison we want to understand is summarized in the following diagram in the case of weight two (and therefore trivial coefficients), where the upper row gives the Beilinson--Flach class and the bottom one, the Belinson--Kato element.
\[\xymatrix{
		H_{\et}^1(Y, \mathbb Z_p(1)) \ar[r] \ar[rd]& H_{\et}^3(Y^2,\mathbb Z_p(2)) \ar[r] & H^1(\mathbb Q, H_{\et}^2(\bar Y^2,\mathbb Z_p(2))) \ar[r] & H^1(\mathbb Q, T_{f,Y} \otimes T_{g,Y}) \\
		& H_{\et}^2(Y,\mathbb Z_p(2)) \ar[r] & H^1(\mathbb Q,H_{\et}^1(\bar Y,\mathbb Z_p(2))) \ar[r] & H^1(\mathbb Q,T_{f,Y}(1)).
	}  \]

When the cohomology class $\bar \kappa_{f,g,1}$ vanishes we may lift the class $\kappa_{f,g}$ to an element in the cohomology with compact support and use the opposite projection.

\begin{theorem}
Suppose that $\bar \kappa_{f,g,1}=0$ and that the conditions of Assumption \ref{ass:bf1} hold. Then the following congruence holds in $H^1(\mathbb Q, T_{f,Y} \otimes \mathcal O/\mathfrak p^t (\ell-s))$:
\[ \bar \kappa_{f,g,2} = \bar \kappa_f. \]
\end{theorem}

\begin{proof}
From the reciprocity laws stated as Proposition \ref{perrin1} and Proposition \ref{perrin2}, we have that \[ \mathcal L_f^-(\res_p(\kappa_{f,\infty})^-) \equiv \mathcal L_{f,g}^{-+}(\res_p(\kappa_{f,g,\infty})^{-+}) \pmod{\mathfrak p^t} \] because both $p$-adic $L$-functions agree. Using Lemma \ref{lemma:inj-kato} we can upgrade this to the desired equality using the period relations discussed in \cite[Section 3.5]{RiRo3}, and the result follows from the assumptions of the statement.
Since the Greenberg subspace of $H^1(\mathbb Q, T_{f,Y}(\ell-s))$ is torsion-free of rank 1, we conclude that the Greenberg module in the right has rank 1 over $\mathcal O/\mathfrak p^t \mathcal O$.
\end{proof}

\begin{remark}
A similar approach works for extending our results to the Iwasawa cohomology classes, that is, incorporating the cyclotomic variation.
\end{remark}

\section{Congruences  between diagonal cycles and Beilinson--Flach elements}

We continue the development of the theory of Eisenstein congruences between Euler systems with the case of diagonal cycles and Beilinson--Flach elements. Consider now a triple of three cuspidal eigenforms $(f,g,h)$ and nebentype $(\chi_f, \chi_g, \chi_h)$, with common level $N$ and with $p \nmid N$. We assume that they satisfy the self-duality condition $\chi_f \chi_g \chi_h=1$, and for simplicity we assume that $\chi_h \not\equiv 1 \pmod{\mathfrak p}$.

The diagonal cycles class we consider is a global element \[ \kappa_{f,g,h} \in H^1(\mathbb Q, T_{f,Y} \otimes T_{g,Y} \otimes T_{h,Y}(1-c)). \]

In this section, we carry out a comparison with the Beilinson--Flach element \[ \kappa_{f,g} \in H^1(\mathbb Q, T_{f,Y} \otimes T_{g,Y}(m-c)). \]

Throughout all the section, we keep the following simplifying assumptions:
\begin{itemize}
\item The {\it weak-Gorenstein condition} for $h$, as presented in Assumption \ref{ass:goren}.
\item The residual Galois representations attached to $f$ and $g$ are absolutely irreducible, and the class $\kappa_{f,g,h}$ is integral at $p$.
\end{itemize}

We consider the Eisenstein congruence setting $h \equiv E_m(\chi_h,1)$. Then there is a class \[ \bar \kappa_{f,g,h,1} \in H^1(\mathbb Q, T_{f,Y} \otimes T_{g,Y} \otimes \mathcal O/\mathfrak p^t(\chi_h)(1-c)). \]

When $\bar \kappa_{f,g,h,1}$ vanishes, and as discussed before, we may consider a putative refinement of it, namely \[ \bar \kappa_{f,g,h,2} \in H^1(\mathbb Q, T_{f,Y} \otimes T_{g,Y} \otimes \mathcal O/\mathfrak p^t(m-c)). \]

We begin with the following assumptions.

\begin{ass}\label{ass:cycles0}
\begin{enumerate}
\item[(a)] The space $H_{\Gr}^1(\mathbb Q, T_{f,Y} \otimes T_{g,Y}(\chi_h)(1-c))$ is one-dimensional.
\item[(b)] It holds that \[ H^0(\mathbb Q, T_{f,Y} \otimes T_{g,Y}(\chi_h)(1-c) \otimes \mathcal O/\mathfrak p^t) = 0, \quad H^2(\mathbb Q, T_{f,Y} \otimes T_{g,Y}(\chi_h)(1-c)) = 0 \] and $\alpha_f \alpha_g^{-1} \chi_h(p), \alpha_f^{-1} \alpha_g \chi_h(p) \not \equiv 1 \pmod{\mathfrak p}$.
\end{enumerate}
\end{ass}

\begin{ass}\label{ass:cycles1}
\begin{enumerate}
\item[(a)] The space $H_{\Gr}^1(\mathbb Q, T_{f,Y} \otimes T_{g,Y}(m-c))$ is one-dimensional.
\item[(b)] It holds that \[ H^0(\mathbb Q, T_{f,Y} \otimes T_{g,Y} \otimes \mathcal O/\mathfrak p^t(m-c)) = 0, \quad H^2(\mathbb Q, T_{f,Y} \otimes T_{g,Y}(m-c)) = 0 \] and $\alpha_f \alpha_g^{-1} \not \equiv 1 \pmod{\mathfrak p}$.
\end{enumerate}
\end{ass}

Kings--Loeffler--Zerbes \cite{KLZ} proved that under big image assumptions for the representation $T_{f,Y} \otimes T_{g,Y}$, condition (a) is true, so this is a relatively mild condition. Assumption (b) may be regarded as a condition for non-exceptional zeros. As before, the previous conditions may be rephrased as a condition over the Frobenius eigenvalues modulo $p$ whenever $p$ is large enough (in particular, $p>k+\ell+m$ is enough):
\begin{itemize}
\item The group $H^0(\mathbb Q_p, T_{f,Y} \otimes T_{g,Y} \otimes \mathcal O/\mathfrak p^t(m-c)) \neq 0$ if and only if $\alpha_f(p) \equiv 1 \pmod{\mathfrak p}$ and $\alpha_g(p) \equiv 1 \pmod{\mathfrak p^t}$.
\item The group $H^2(\mathbb Q_p, T_{f,Y} \otimes T_{g,Y}(\chi_h)(1-c)) = 0$ if and only if $\chi_h(p) \equiv 1 \pmod{\mathfrak p}$.
\item The group $H^0(\mathbb Q_p, T_{f,Y} \otimes T_{g,Y} \otimes \mathcal O/\mathfrak p^t(m-c)) = 0$ if and only if $\chi_h(p) \equiv 1 \pmod{\mathfrak p^t}$.
\item The group $H^2(\mathbb Q, T_{f,Y} \otimes T_{g,Y}(m-c))=0$ if and only if $\chi_f(p) \equiv 1$ or $\chi_g(p) \equiv 1 \pmod{\mathfrak p}$ .
\end{itemize}

\begin{propo}
Under the Assumptions \ref{ass:cycles1}, the space $H_{\Gr}^1(\mathbb Q, T_{f,Y} \otimes T_{g,Y} \otimes \mathcal O/\mathfrak p^t)$ is one-dimensional.
\end{propo}

The following result is proved in the same way as in the previous section using the general theory of Perrin-Riou maps.

\begin{lemma}\label{lemma:inj-bf}
Assume that $H^0(\mathbb Q_p, T_{f,Y} \otimes T_{g,Y}(m-c))=0$. Then the Bloch--Kato logarithm is injective.
\end{lemma}

We now state the first congruence concerning diagonal cycles.

\begin{propo}
Suppose the conditions of Assumption \ref{ass:cycles0} hold. Then if $L_p(f,g,\chi_h,c) \neq 0$ or $L_p(g,f,\chi_h,c) \neq 0$ modulo $p$, the class $\bar \kappa_{f,g,h,1}$ vanishes.
\end{propo}

\begin{proof}
The space $H_{\Gr}^1(\mathbb Q, T_{f,Y} \otimes T_{g,Y}(\chi_h)(1-c) \otimes \mathcal O/\mathfrak p^t)$ is one-dimensional under the running assumptions. It is spanned by the Beilinson--Flach class when this is non-zero; this is the class resulting from the specialization of the element $\kappa_{f,g,\infty}$ in Theorem \ref{bf-law}.

If $L_p(f,g,\chi_h,c) \neq 0$, the projection of this class to $H^1(\mathbb Q_p, T_{f,Y} \otimes T_{g,Y}^- (\chi_h)(1-c) \otimes \mathcal O/\mathfrak p^t)$ is non-zero, since its image under the Perrin-Riou regulator is precisely $L_p(f,g,\chi_h,c)$, which does not vanish due to the assumptions. However, the projection to that quotient of $\bar \kappa_{f,g,h,1}$ is zero because of the local properties satisfied by the diagonal cycle class (see \cite[Prop. 5.8]{DR3}).
\end{proof}

When $\bar \kappa_{f,g,h,1}$ vanishes, we may lift the class to the cohomology of the open modular curve and consider a diagram like the following (in weight 2), where the upper row gives the diagonal cycle and the bottom one, the Belinson--Flach element. (We have shortened the notation and written $T_{f,g,h}:= T_{f,Y} \otimes T_{g,Y} \otimes T_{h,Y}$ and $T_{f,g} := T_{f,Y} \otimes T_{g,Y}$.)

\[\xymatrix{
		H_{\et}^0(Y, \mathbb Z_p) \ar[r] \ar[rd]& H_{\et}^4(Y^3,\mathbb Z_p(2)) \ar[r] & H^1(\mathbb Q, H_{\et}^3(\bar Y^3,\mathbb Z_p(2))) \ar[r] & H^1(\mathbb Q, T_{f,g,h}(-1)) \\
		& H_{\et}^3(Y^2,\mathbb Z_p(2)) \ar[r] & H^1(\mathbb Q,H_{\et}^2(\bar Y^2,\mathbb Z_p(2))) \ar[r] & H^1(\mathbb Q,T_{f,g}).
	}  \]

\begin{theorem}
Suppose that $\bar \kappa_{f,g,h,1}=0$ and that the conditions of Assumption \ref{ass:cycles1} hold. The following congruence holds in $H^1(\mathbb Q, T_{f,g} \otimes \mathcal O/\mathfrak p^t(m-c))$:
\[ \bar \kappa_{f,g,h,2} = \bar \kappa_{f,g}. \]
\end{theorem}

\begin{proof}
We study the congruence via the theory of Perrin-Riou maps. Consider the projection \[ \pi_{f,g}^{-+} \, : \, H^1(\mathbb Q, T_{f,g}(m-c) \otimes \mathcal O/\mathfrak p^t) \longrightarrow H^1(\mathbb Q_p, T_{f,Y,\circ}^{\quo} \otimes T_{g,Y}^{\sub}(m-c) \otimes \mathcal O/\mathfrak p^t). \]

From the previous proposition and the reciprocity laws of \cite{KLZ} and \cite{DR3}, we have that the images of $\bar \kappa_{f,g,h,2}$ and $\kappa_{f,g}$ under the appropriate Bloch--Kato logarithms agree, since the corresponding $p$-adic $L$-functions take the same value; this can be easily seen, for instance, using the expressions in terms of Petersson products, as in \cite[Cor. 4.13]{DR1}. We can upgrade this to the equality of the statement using the period relations discussed in \cite[Section 3.5]{RiRo3}, and the statement follows.

Using now Lemma \ref{lemma:inj-bf}, we may upgrade that result to an equality in $H^1(\mathbb Q, T_{f,Y} \otimes T_{g,Y} \otimes \mathcal O/\mathfrak p^t(m-c))$.
\end{proof}

\begin{remark}
It would be interesting to understand these congruences in terms of the explicit geometric description of the cycles, following the seminal work of Bloch \cite{Bl} and the approach of \cite{Mot}. This approach, as well as others connecting different Euler systems, is based on a comparison at the level of $p$-adic $L$-functions, and it would be nice to have a more geometric understanding of this phenomenon.
\end{remark}

\begin{remark}
A case which is specially intriguing occurs when more than one of the modular forms are congruent with Eisenstein series. However, the assumptions needed to conclude that the corresponding Greenberg cohomology spaces are one-dimensional do not hold (the big image condition does not hold since the modulo $\mathfrak p$ representations are not irreducible). We expect to end up with a single class in $H^1(\mathbb Q, \mathcal O/\mathfrak p^t(2))$, which could potentially be understood using similar techniques to those developed by Sharifi and Fukaya--Kato.
\end{remark}

\end{document}